\documentclass[11pt,a4paper, twoside]{amsart}
\usepackage[utf8]{inputenc}

\usepackage{tikz}

\usepackage{amsmath}
\usepackage{amsfonts}
\usepackage{amssymb}
\usepackage{latexsym}

\usepackage{hyperref}
\hypersetup{backref, pdfpagemode=FullScreen, colorlinks=true,
  citecolor=magenta, linkcolor=cyan, urlcolor=blue}

\usepackage{mathtools}
\mathtoolsset{showonlyrefs}

\usepackage{graphicx}
\usepackage{xcolor}
\usepackage{graphicx} 

\usepackage{version}
\usepackage{todonotes}

\theoremstyle{change}
\newtheorem{theo}{Theorem}[section]
\newtheorem*{theo*}{Theorem}
\newtheorem{lemma}[theo]{Lemma}
\newtheorem{corollary}[theo]{Corollary}
\newtheorem{prop}[theo]{Proposition}

\newtheorem{rem}[theo]{Remark}

\newtheorem{definition}[theo]{Definition}
\newtheorem{theoman}{Theorem}[section]

\newcommand{\vol}{\mathrm{vol}\,}

\newcommand{\ip}[2]{\left\langle #1,#2\right\rangle}
\newcommand{\bnorm}[1]{\left\vert#1\right\vert}

\newcommand{\Kn}{{\mathcal K}^n} 
\newcommand{\Ksn}{{\mathcal K}_{e}^n} 
\newcommand{\Knull}{{\mathcal K}_{(o)}^n} 
\newcommand{\Kcn}{{\mathcal K}_c^n} 
\newcommand{\HomF}[1]{{H(#1,n)}}

\newcommand{\R}{\mathbb{R}} 
\newcommand{\N}{\mathbb{N}} 
\newcommand{\Rn}{\mathbb{R}^n} 

\newcommand{\rat}{\mathcal{I}} 

\newcommand{\inte}{\mathrm{int}\,}  
\newcommand{\ov}{\overline}

\newcommand{\dint}{\mathrm{d}}
\newcommand{\sph}{\mathbb{S}}

\newcommand{\dcura}{\widetilde{\mathrm{C}}}
\newcommand{\haus}{\mathcal{H}}

\newcommand{\V}{\mathrm{V}}

\numberwithin{equation}{section}

\begin{document}
	
\title[Subspace concentration   for dual curvature measures]{On
  subspace concentration  for  dual curvature measures}
\author{Katharina Eller}\author{Martin Henk}
\address{Institut f\"ur Mathematik, Technische Universit\"at Berlin,
	Sekr. Ma 4-1, Strasse des 17 Juni 136, D-10623 Berlin, Germany}
\email{\{eller, henk\}@math.tu-berlin.de}
	
\begin{abstract}
We study   subspace concentration of 
dual curvature measures of convex bodies $K$ satisfying $\gamma (-K)
\subseteq K$ for some $\gamma \in (0,1]$. We present upper bounds on the subspace
concentration depending  on $\gamma$, which, in particular, retrieves
the known results in the symmetric setting. The proof is based on a unified approach to prove necessary subspace concentration conditions via the divergence theorem.
\end{abstract}
\maketitle

\section{Introduction}

Let $\Kn$ denote the set of convex bodies in $\R^n$,
i.e., the family of all convex and compact subsets
$K\subset\R^n$ with non-empty interior. The subfamily of convex bodies 
containing the origin in their interior, i.e., $0 \in
\operatorname{int}K$ is denoted by $\Knull$ 
and the subset of origin-symmetric convex bodies, i.e., the sets $K \in
\Kn$ satisfying $K=-K$,
is denoted by $\Ksn$. A convex body $K$ is called centered if its
centroid is located at the origin, i.e.,
\begin{equation*}
   \frac{1}{\vol(K)}\int_K x \, \dint \mathcal{H}^{n}(x)=0,
\end{equation*}
where, in general, $\haus^k$ denotes 
 the  $k$-dimensional Hausdorff measure, and when referring to the
$n$-dimensional volume we will write $\vol$ instead of
$\mathcal{H}^n$. The set of all  centered convex bodies in $\R^n$ is denoted by
$\Kcn$, and, in particular, we have $\Ksn\subset \Kcn\subset\Knull$.

As usual, for $x,y \in \R^n$ let $\ip{x}{y}$ denote the
standard inner product on $\R^n$, and $\bnorm x=\sqrt{\ip{x}{x}}$ the Euclidean norm of $x$.  We write
$B_n$ for the $n$-dimensional Euclidean unit ball, i.e., 
$B_n=\{ x\in\R^n : \bnorm{x} \leq 1 \}$, and $\sph^{n-1}=\partial
B_n$, where $\partial A$ is the set of boundary points of a set $A\subset\R^n$.

There are two  far-reaching extensions of the classical Brunn-Minkowski
theory, the $L_p$-Brunn-Minkowski theory and the dual Brunn-Minkowski
theory. Both of them are cornerstones of modern convex geometry and
both of them arise, roughly
speaking, by studying the volume of the sum of convex
bodies, where the usual Minkowski addition for building the sum is replaced by
another kind of addition. In the case of the $L_p$-Brunn-Minkowski
theory this is the so called $L_p$-addition, introduced by Firey
\cite{firey1962p} and Lutwak \cite{lutwak1993brunn,lutwak1993selected,lutwak1996brunn} for which we also refer to \cite[Section 9.1,
9.2]{schneider_2013}. In the dual Brunn-Minkowski theory the
so called radial addition, introduced by Lutwak \cite{lutwak1975dual},
is used (see also \cite[Section 9.3]{schneider_2013}).

One of the central problems in classical Brunn-Minkowski theory is the
 {\em Minkowski-Christoffel} problem asking for necessary and
sufficient conditions characterizing the surface area measures of a convex body among the finite Borel
measures on the sphere. For a definition of these surface area measures
and on the state of the art of the Minkowski-Christoffel problem we refer to
\cite[Chapter 8]{schneider_2013}. 

In the ground-breaking paper \cite{huang2016geometric} by Huang, Lutwak, Yang and
Zhang, the missing ``dual'' counterparts to these surfaces area measures within
the dual Brunn-Minkowski theory were introduced. They are called dual
curvature measures. In contrast to the
surface area measures, they admit an explicit integral
representation. To this end,  for $K\in\Knull$ let $\rho_K$ be the radial function, i.e., for
$x\in\R^n\setminus\{0\}$ let
\begin{equation*}
    \rho_K(x)=\max\{\rho>0 : \rho\,x\in K\}.
\end{equation*}   
Then for $q\in\R$, the $q$-th dual curvature measure of $K$ is a finite
Borel measure on $\sph^{n-1}$ given by 
\begin{equation*}
    \dcura_{K,q}(\eta)=\frac{1}{n}\int_{ \alpha^{\ast}_K(\eta)} \rho_K(u)^q \, d\mathcal{H}^{n-1}(u),
  \end{equation*}
  where for a Borel set $\eta\subseteq\sph^{n-1}$, the set  
 $\alpha^{\ast}_K(\eta)$ consists of all $u \in \sph^{n-1}$ such that
 the boundary point $\rho_K(u)u$ of $K$ has an  outer unit normal vector
 in $\eta$.

 In analogy to the above  mentioned classical Minkowski-Christoffel problem,
 the {\em dual Minkowski problem}, posed by Huang et al.~in \cite{huang2016geometric}, asks for
 necessary and sufficient conditions when a finite Borel measure $\mu$ on
 the sphere is the $q$-th dual
 curvature measure of a convex body $K\in\Knull$.

 Among these dual curvature measures there are two particular
 important measures. The $0$-th  dual curvature measure coincides up to
 a constant with  Alexandrov's integral curvature measure of the polar
 body of $K$, and the corresponding Minkowski problem, known as
 Alexandrov problem  has been solved
 by Alexandrov \cite{aleksandrov1942existence}. For extensions to the
 $L_p$ setting of the Alexandrov problem we refer to
 \cite{HuangLutwakYangEtAl2018, MR4455361} and the references within.

 The $n$-th dual curvature measure is in fact
 the cone volume measure $\V_K$ of $K$ that is
\begin{align*}
    \dcura_{K,n}(\eta)=\V_K(\eta)= \frac{1}{n} \int_{\nu_K^{-1}(\eta)}
  \ip{\nu_K(u)}{u} \dint\haus^{n-1}(u),
\end{align*}
where $\nu_K(\cdot)$ is the spherical image map (see Section 2),
essentially the Gauss map on the regular boundary points of $K$.
The characterization of the cone volume measure is known as the {\em logarithmic Minkowski
problem}. It has been studied extensively over the last few years in
many different contexts, see, e.g., \cite{ Boroczky2022,MR4286679, MR3369507, MR3509941, boroczky2016cone, boroczky2014conevolume, MR3037788,
  MR3896091,
  henk2014cone,MR4060313, MR1901250, MR2019226,MR2729006}, 
 and
for results in the general $L_p$ setting see, e.g.,
\cite{
  MR4058533, ChouWang2006a, GuoXiZhao2023, HugLutwakYangEtAl2005}.  

Regarding the dual Minkowski problem there is an obvious necessary
condition, namely the measure $\mu$ must not be concentrated on any
closed hemisphere of $\sph^{n-1}$. For $q<0$ this is surprisingly also sufficient as
shown by Yiming Zhao \cite{zhao2017dual}. For positive parameters $q$ the
behaviour  seems to be different and a quantitative "subspace
concentration" appears. In order to describe it, we set  
\begin{equation*}
                \rat(\mu,L)=\frac{\mu(\sph^{n-1}\cap L)}{\mu(\sph^{n-1})},
\end{equation*}
for a linear subspace $L\subset \R^n$, $\dim
L\geq 1$, and a non-zero finite Borel measure $\mu$ on $\sph^{n-1}$. Due to the joint efforts of Böröczky, Henk,
Huang, Lutwak, Pollehn, Yang,
Zhang, Zhao, \cite{boroczky2018subspace, boroczky2019dual, huang2016geometric, zhao2018existence},
a complete solution of the dual Minkowski
problem in the {\em even} case is known in the range $q\in(0,n)$. 
\begin{theoman}[Theorem 1.1, \cite{boroczky2019dual}] Let
  $q\in(0,n)$, and let  $\mu$ be  an {\em even} non-zero finite  Borel
  measure on $\sph^{n-1}$. Then there exists a convex body $K\in\Ksn$ such that
  $\mu=\dcura_{K,q}$ if and only if  
  \begin{equation}
    \rat(\mu,L) < \min\left\{\frac{\dim L}{q},1\right\}
   \label{eq:smi} 
  \end{equation}
  for all proper linear subspaces $L\subset\R^n$.
\label{thm:evensmallq}  
\end{theoman}
For $q=n$, i.e., for the log-Minkowski problem a complete solution in the even case was given by Böröczky, Lutwak,
Yang and Zhang.
\begin{theoman}[Theorem 1.1, \cite{boroczky2013logarithmic}] \label{thmmain: sci cone volume} Let  $\mu$ be  an {\em even} non-zero finite  Borel
  measure on $\sph^{n-1}$. Then there exists a convex body $K\in\Ksn$ such that
  $\mu=\dcura_{K,n}$ if and only if
  \begin{equation}
                 \rat(\mu,L) \leq  \frac{\dim L}{n}, 
   \label{eq:scc}
  \end{equation}
  for all proper linear subspaces $L\subset\R^n$, and whenever equality
  holds in \eqref{eq:scc} for some $L$ then there exists a complementary subspace
  $L'$ such that $\mu$ is concentrated on $(L\cup
  L')\cap \sph^{n-1}$.
\label{thm:conesym}  
\end{theoman}
For $q>n$ the dual (even) Minkowski problem is open, some necessary
conditions are known, however, at least for $q>n+1$.   
\begin{theoman}[Theorem 1.7, \cite{henk2017necessary}]  Let  $q>n+1$ and
  $K\in\Ksn$. Then  
  \begin{equation}
                 \rat(\dcura_{K,q},L) <  \frac{\dim L+q-n}{q}, 
                 \label{eq:smin}
  \end{equation}
  for all proper linear subspaces $L\subset\R^n$.
\label{thm:evenlargeq}    
\end{theoman}
This inequality is best possible 
and it is likely to be sufficient as well. For $n=2$, \eqref{eq:smin} holds
even true for $q>2$.

In the non-even case we know only very little for $q>0$. In fact, only
the case $q=n$ (cone-volume measure) has been studied in this respect
and even there  we do not have matching necessary and sufficient
conditions. For centered convex bodies it was shown by Böröczky and
Henk \cite{boroczky2016cone} (see also \cite{henk2014cone} for the polytopal case) that \eqref{eq:scc} is also necessary.
\begin{theoman}[Theorem 1.3, \cite{boroczky2016cone}] \label{thm:conecentered} Let  $K\in\Kcn$. Then
  \begin{equation*}
                 \rat(\dcura_{K,n},L) \leq  \frac{\dim L}{n}, 
  \end{equation*}
  for all proper linear subspaces $L\subset\R^n$, and whenever equality
  holds for some $L$ then there exists a complementary subspace
  $L'$ such that $\mu$ is concentrated on $(L\cup
  L')\cap \sph^{n-1}$.
\end{theoman}

The proof of the necessity of the inequalities in Theorems
\ref{thm:evensmallq}, \ref{thm:evenlargeq}, \ref{thm:conecentered} are
based on three different approaches.  
The main purpose  of this paper is i) to unify these approaches
and ii)  based on this unification to establish  first results on the
subspace concentration of the dual curvature measures of arbitrary
bodies  $K\in\Knull$.    

\begin{theo}\label{theo:maingeneralbodies}
Let $K \in \Knull $, $\gamma \in (0,1]$ such that $\gamma (-K) \subseteq
K$. Let $L \subset \mathbb{R}^n$ be a proper subspace and let
$q\in\R$ with  $q>\dim L+1$.  Then
\begin{equation*}
  \rat(\dcura_{K,q},L) \leq
  \begin{cases}
    &\frac{\dim(L)+\frac{1-\gamma}{1+\gamma}(q-\dim(L))}{q}, q\leq n, 
    \\
    &\frac{(\dim(L)+q-n)+ \frac{1-\gamma}{1+\gamma}(n-\dim(L))}{q}, q>n+1.
   \end{cases} 
\end{equation*}   
\end{theo}
For $K\in\Ksn$, i.e., $\gamma=1$, this theorem implies essentially the necessity parts of
Theorem \ref{thm:evensmallq} and Theorem \ref{thm:evenlargeq}. The
additional restriction $q>\dim L+1$ (instead of $q>\dim L$) in the
range $q\leq n$ is caused by our more general approach, but is likely to be not necessary. As for centered
convex bodies $K\in\Kcn$, the asymmetry parameter $\gamma$ in the
theorem above may be chosen to be at least $1/n$ (cf. \cite{hammer1951centroid, suss1948affininvariante}) we get as a corollary
\begin{corollary} Let $K \in \Kcn$. Let $L \subset \mathbb{R}^n$ be a proper subspace and let
$q\in\R$ with  $q>\dim L+1$.  Then
\begin{equation*}
  \rat(\dcura_{K,q},L) \leq
  \begin{cases}
    &\frac{\dim(L)+\frac{n-1}{n+1}(q-\dim(L))}{q}, q\leq n, 
    \\
    &\frac{(\dim(L)+q-n)+ \frac{n-1}{n+1}(n-\dim(L))}{q}, q>n+1.
   \end{cases} 
\end{equation*}   
\end{corollary}
Numerical results indicate that for $K\in\Kcn$ the
same inequalities hold true as in the even case.
With our approach, as we will see this amounts to control a certain
integral of directional derivatives, which we can handle efficiently only in case
$q=n$ leading to Theorem \ref{thm:conecentered}.

In order to describe our approach which is based on
\cite{boroczky2016cone, boroczky2018subspace, henk2017necessary} we need some
more notation. For $K\in\Knull$, $L\subset\R^n$ a proper subspace and
$q\in\R$ with $q>\dim L$ let
\begin{equation}\label{eq:deffunctiong}
     g_{K,L,q}(x) = \int_{K \cap (x+L^{\perp})} \bnorm{z}^{q-n} \, \dint \haus^{n-\dim L}(z)
\end{equation}
where $x\in K|L$, i.e., $x$ belongs to the orthogonal projection
of $K$ onto $L$, and $L^\perp$ is the orthogonal complement of $L$.
 For $q<n$ the integrand displays a singularity at the origin and is
unbounded. However as long as we require $q>\dim L$ the integral exists.

By applying a generalized divergence theorem from \cite{pfeffer2012divergence} and establishing
regularity properties of the section function $g_{K,L,q}$ we will
show
\begin{theo}\label{thm:Gaussdiv}
Let $K \in \Knull$, $L\subset \R^n$ be a proper subspace and let $q\in\R$
with $q > \dim L+1$.
Then it holds
\begin{equation*}
 \rat(\dcura_{K,q},L)  = \frac{\dim L}{q} +\frac{1}{n}\frac{1}{ \dcura_{K,q}(\sph^{n-1})}
 \int_{K \vert L}\ip{\nabla
   g_{K,L,q}(x)}{x} \, \dint \haus^{\dim L}(x).
\end{equation*}
\end{theo} 
For $q=n$ the  function $g_{K,L,n}$ is log-concave and based on this 
property it was shown  in \cite{boroczky2016cone} that  $\int_{K \vert L}\ip{\nabla
   g_{K,L,n}(x)}{x} \, \dint \haus^{\dim L}(x)\leq 0$. This implies
 Theorem \ref{thm:conecentered} except for the range $ \dim L+1 \geq
 q>\dim L$.  For $q\ne n$, however, the slicing function is not log-concave and thus
 behaves quite differently.

 Theorem \ref{theo:maingeneralbodies}
   will follow immediately from Theorem \ref{thm:Gaussdiv} and the
 following bounds
 \begin{theo} Let $K \in \Knull $, $\gamma \in (0,1]$ such that $\gamma(- K) \subseteq
K$. Let $L \subset \mathbb{R}^n$ be a proper subspace and let
$q\in\R$ with  $q>\dim L+1$.  Then it holds
\begin{align*}
 \frac{1}{n  \dcura_{K,q}(\sph^{n-1})}\int_{K \vert L}&\ip{\nabla
   g_{K,L,q}(x)}{x} \, \dint \haus^{\dim L}(x) \\ \leq 
  &\begin{cases}
    &\frac{q-\dim(L)}{q}\frac{1-\gamma}{1+\gamma}, q\leq n, 
    \\[1ex]
    &\frac{(q-n)+\frac{1-\gamma}{1+\gamma}(n- \dim(L))}{q}, q>n+1.
   \end{cases} 
 \end{align*}
\label{thm:boundssection} 
\end{theo}

For results on the dual Minkowski problem in the smooth setting we
refer to \cite{MR3818073, MR3659361, MR4055992} and the references within. 
The paper is organized as follows: Necessary notation and
preliminaries from Convex Geometry will be given in Section \ref{sec 2}. The
proof of Theorem \ref{thm:Gaussdiv} is presented in Section \ref{sec : 3} where we
actually prove a result for a slightly larger class of functions than
$g_{K,L,q}(x)$ (see Theorem \ref{thm: Gauss div gen}). Section \ref{sec : 4} is devoted to the proof of Theorem \ref{thm:boundssection}
and thus of Theorem \ref{theo:maingeneralbodies}.

\section{Preliminaries and Notation}\label{sec 2}
We begin with a few basic facts about convex bodies and functions for which we refer to
\cite{ gardner1995geometric, gruber2007convex, rockafellar1970convex, schneider_2013}. 
A function $f :\mathbb{R}^n \to \mathbb{R}_{\geq 0}$ is called quasiconcave (or unimodal) if $f((1-\lambda)x+\lambda y) \geq \min \{f(x),f(y)\}$ holds for all $x,y \in \R^n$.

As usual, a function $f:A\to\R^m$, $A\subseteq \R^n$, is called
\emph{Lipschitz continuous} or just
\emph{Lipschitz}  if there exists a constant $L\geq 0$ such that for all $x,y\in A$
\begin{equation*}
  \bnorm{f(x)-f(y)}\leq L\bnorm{x-y}.
\end{equation*}
 A function $f: A \to \mathbb{R}^m$ will be called  
 \emph{locally Lipschitz}  if for every $x\in A$ there
 exists an open neighbourhood $U\subseteq A$ such that $f_{|U}$ is
 Lipschitz.  By a standard compactness argument we have that 
 a locally Lipschitz function $f:A\to\R^m$   is Lipschitz
 on all compact subsets of $A$.


For $p\in\R$ we denote by  $\HomF{p}$ the class of all
functions $f : \mathbb{R}^n\setminus\{0\} \to \mathbb{R}_{\geq 0}$
which are \textit{positively homogeneous} of degree $p$, i.e., for all
$x\in\R^n\setminus\{0\}$ and $\alpha>0$ we have
\begin{equation*}
   f(\alpha x)=\alpha^p f(x).
\end{equation*}   
Observe for $p<0$ and $f\in \HomF{p}$ we must have $\lim_{x\to 0} f(x)=\infty$.
The next proposition states the fact that  a function $f \in \HomF{p}$
which is Lipschitz restricted to the sphere $\sph^{n-1}$ is locally
Lipschitz on $ \mathbb{R}^n \setminus \{0\}$. 
We will state this
fact in a rather explicit form for later purpose.

\begin{lemma} \label{lem: varphi loc Lipsch}
Let $p\in\R$, $f\in\HomF{p}$ be Lipschitz  on $\sph^{n-1}$. Then $f$
is locally Lipschitz  on $ \mathbb{R}^n \setminus
\{0\}$. More precisely, let $a\in  \mathbb{R}^n \setminus
\{0\}$. Then for all $x,y\in a+\frac{1}{2} \vert a \vert B_n$ we have
\begin{equation*}
                \vert f(x)-f(y)\vert  \leq c_{f}  \bnorm{a}^{p-1}\bnorm{x-y},
\end{equation*}  
where $c_{f}$ is a  constant depending only on $f$.
\end{lemma}
\begin{proof} 
  First we observe that the 
 power functions of the norm are locally Lipschitz on
 $\mathbb{R}^n \setminus \{0\}$. To this end we note that 
 \begin{equation*}
         \max_{z\in a+\frac{1}{2} \vert a \vert B_n}
     \bnorm{z}^{p-1} = c_p  \bnorm{a}^{p-1}
 \end{equation*}
where $c_p$ is a constant depending only on $p$.  For $p\ne 0$ we get
by the  mean value theorem for $x,y\in a+\frac{1}{2} \vert a \vert B_n$
\begin{equation}
  \begin{split} 
   |\bnorm{x}^p-\bnorm{y}^p| & \leq |p|\left(\max_{z\in a+\frac{1}{2} \vert a \vert B_n}
     \bnorm{z}^{p-1}\right)\bnorm{x-y}\\ 
   &=\bar c_p \bnorm{a}^{p-1} \bnorm{x-y},
 \end{split}
 \label{eq:lip_power_norm}
 \end{equation}
for a constant $\bar c_p$ depending only on $p$. For $p=0$ we set
$\bar c_p=0$ and the inequality is certainly still true.


Now let  $\overline{z}=z/\bnorm{z}\in\sph^{n-1}$ for $z
\in\R^n\setminus\{0\}$ and,  moreover, let
$\alpha=\max\{f(z):z\in\sph^{n-1}\}$. By assumption there exists a
constant $L$ such that $|f(\overline x)-f(\overline y)|\leq L
\bnorm{\overline x-\overline y}$ for all
$x,y\in\R^n\setminus\{0\}$. As the convex function
$\bnorm{t\overline x-\overline y}^2$, $t\in\R$, is minimal
at $t=\langle \overline x, \overline y\rangle\leq 1$ we conclude 
for $\bnorm{x}\geq\bnorm{y}$ that 
\begin{equation}
     |f(\overline x)-f(\overline y)|\leq L
\bnorm{\overline x-\overline y}\leq L
\bnorm{\frac{\bnorm{x}}{\bnorm{y}}\overline{x}-\overline{y}}=L\frac{1}{\bnorm{y}}\bnorm{x-y}.
\label{eq:lipschitz_sphere}
\end{equation}   
 
Then for  $x,y\in a+\frac{1}{2} \vert a \vert B_n$, $\bnorm{x}\geq\bnorm{y}$,  we may
write in view of \eqref{eq:lip_power_norm} and \eqref{eq:lipschitz_sphere}
\begin{align*}
    \vert f(x) -f(y) \vert &= \left\vert
                             \bnorm{x}^{p}f(\overline{x}) -
                              \bnorm{y}^{p}f(\overline{y})
                             \right\vert \\
    & \leq\bnorm{y}^{p} \left\vert f(\overline{x})
      -f(\overline{y}) \right\vert + \left\vert
      f(\overline{x}) \right\vert \vert \bnorm{x}^{p} -
      \bnorm{y}^{p} \vert \\
&\leq \bnorm{y}^{p-1}L  \bnorm{x-y} +f(\overline{x}) \bar c_p
                                                           \bnorm{a}^{p-1}
                                                           \bnorm{x-y}\\
 &\leq \left(c_p\,L+\alpha\bar c_p\right) \bnorm{a}^{p-1}
      \bnorm{x-y}.
\end{align*}
With $c_f=c_p\,L+\alpha\bar c_p$ the assertion follows.
\end{proof}

For a given convex body $K \in \Kn$ the \emph{support function} $h_K:
\mathbb{R}^n \to \mathbb{R}$
is defined by
\begin{equation*} 
h_K(u)=\max_{x \in K} \ip{u}{x}.
\end{equation*}
The support function is convex, continuous and, in particular, $h_K\in
\HomF{1}$. The hyperplane
\begin{equation*}
  H_K(u)=\{x \in \mathbb{R}^n :   \ip{u}{x}=h_K(u)\}
\end{equation*} 
is a supporting hyperplane of $K$ and for a boundary
point $v \in \partial K\cap H_K(u)$, the vector $u$ will be called an
\emph{outer normal vector}. If in addition $u\in \sph^{n-1}$ then is an
\emph{outer unit normal vector}. Let $\partial
^{\ast} K\subseteq \partial K$ be the set of all boundary points
having an unique outer unit normal vector.
We remark that  the set of boundary points not having an unique outer
normal vector has measure zero, that is $\mathcal{H}^{n-1}(\partial K
\setminus \partial ^{\ast} K )=0$.

The \emph{spherical image map}
$\nu_K: \partial^{\ast} K \to \mathbb{S}^{n-1}$ maps a point $x$ to
its unique outer unit normal vector.

A kind of dual counterpart to the support function is the \emph{radial
function} $\rho_K:
\mathbb{R}^n\setminus \{0\}\to \mathbb{R}$ for $K \in \Knull$. It is
given by
\begin{equation}
  \label{eq : radial function}
  \rho_{K}(u)=\max \{\rho>0    : \rho u \in K\}.
\end{equation}   
The radial function $\rho_K$ is positive, continuous, Lipschitz on
$\sph^{n-1}$ with respect to the Euclidean metric, quasiconcave and
$\rho_k\in \HomF{-1}$.

Let $\Omega\subset \mathbb{S}^{n-1}$ be the set of all unit vectors such that for
$u\in\Omega$ the boundary
point $\rho_K(u)\,u$ has an unique outer normal vector. The map
$\alpha_K: \Omega \to \sph^{n-1}$ with  $\alpha_K(u)=\nu_K(\rho_K(u)u)$ 
is called the \emph{radial Gauss map}.  
For $\eta \subseteq \sph^{n-1}$, the \textit{reverse radial Gauss
  image} of $\eta$ is defined by \[\alpha_K^{\ast}(\eta)= \{u \in
  \mathbb{S}^{n-1} \, \vert \, \rho_K(u)u \in H_K(v) \text{ for some }
  v \in \eta \}.\] The reverse radial Gauss image of $\eta$ consists
of all $u \in \mathbb{S}^{n-1}$ such that the boundary point
$\rho_K(u)u$ has an outer unit normal vector in $\eta$ (see, e.g., \cite{huang2016geometric}).  

The maximal Euclidean distance between two points of $K$, i.e., the diameter of $K$, is denoted by $D(K)$, and for $A,B\subset\R^n$
\begin{align*}
    d(A,B)=\inf \{\delta>0 \, \vert \, A \subseteq B + \delta B_{n} \text{ and } B \subseteq A + \delta B_{n}\}
\end{align*}
denotes the Hausdorff distance between $A$ and $B$.

\section{Proof of Theorem \ref{thm:Gaussdiv}}\label{sec : 3}
As mentioned before, our proof strategy for Theorem \ref{thm:Gaussdiv}
allows for a slightly extended definition of the dual curvature measure depending on a function $\varphi$ satisfying homogeneity and Lipschitz continuity. Inspired by previous generalizations as for example in \cite{lutwak2018lp} we present the results of this section in this generalized form.
\begin{definition}\label{defi : extended dual curvature measure}
  Let $K \in \Knull$, $q\in \mathbb{R}$, and let $\varphi \in
  \HomF{q-n}$ be Lipschitz continuous on $\sph^{n-1}$. 
  For a Borel set $\eta \subseteq \sph^{n-1}$ let 
\begin{align*}
    \tilde{C}_{K,\varphi,q}(\eta)= \frac{1}{n}\int_{\alpha_K^{\ast}(\eta)} \varphi (u) \rho_K(u)^q \, du.
\end{align*}
\end{definition}
 Obviously, for $\varphi=|\cdot|^{q-n}=\rho_{B_n}^{n-q}$ we get the dual curvature measure. Without assuming homogeneity and Lipschitz continuity we draw the connection to existing definitions: For $\varphi=\rho_Q^{n-q}$ the $q$-th dual curvature measure with star body $Q$ introduced in \cite{lutwak2018lp} by Lutwak, Yang and Zhang is recovered. For $\varphi=(h_K \circ \alpha_K)^{-p}\rho_Q^{n-q}$ the $L_p$ dual curvature measure also introduced in \cite{lutwak2018lp} is retrieved. 
 For sake of completeness we mention that there exists an even more
 general definition namely the general dual Orlicz curvature measure
 introduced and  examined in \cite{gardner2019general,gardner2020general}. However, our definition is tailored to the new approach presented.

In analogy to \cite[Lemma 2.1]{boroczky2018subspace} we can express $\tilde{C}_{K,\varphi,q}(\eta)$ for $q>0$ as an integral of the function $\varphi(x)$.
\begin{lemma} \label{lem : generalized dcm representation}
Let $K \in \Knull$,  $q > 0$ and $\varphi \in
  \HomF{q-n}$. Then 
\begin{equation} 
  \dcura_{K,\varphi,q}( \eta) = \frac{q}{n} \int_{\left\{x \in K : x/\vert x \vert \in \alpha^{\ast}_K(\eta)\right\}} \varphi (x) \, d \mathcal{H}^{n}(x).
\label{eq:integral_extended}   
\end{equation}
\end{lemma}
\begin{proof} As in \cite{boroczky2018subspace} one obtains by using spherical coordinates, Definition \ref{defi : extended dual curvature measure} and the homogeneity of $\varphi$:
\begin{align*}
     \frac{q}{n} \int_{\left\{x \in K : x/\vert x \vert \in \alpha^{\ast}_K(\eta)\right\}} \varphi (x) \, d \mathcal{H}^{n}(x) &= \frac{q}{n} \int_{ \alpha^{\ast}_K(\eta)} \int_0^{\rho_K(u)} r^{n-1} \varphi (ru) \, dr\, d u \\ &= \frac{q}{n} \int_{ \alpha^{\ast}_K(\eta)}\varphi(u)\int_0^{\rho_K(u)} r^{n-1+q-n}  \, dr\, d u \\
     &= \frac{1}{n} \int_{ \alpha^{\ast}_K(\eta)}\varphi(u)\rho_K(u)^{q}  \, du \\
     &= \tilde{C}_{K,\varphi,q}(\eta). \qedhere
\end{align*}
\end{proof}   

Since we are going to  evaluate the integral in
\eqref{eq:integral_extended} along slices of $K$ with affine planes we
set for $x\in\R^n$, $q>0$ and $\varphi \in
  \HomF{q-n}$  
\begin{equation*} 
    g_{K,L,\varphi,q}(x) = \int_{K \cap (x+L^{\perp})} \varphi(z) \, d \mathcal{H}^{n-\dim L}(z), 
\end{equation*}   
where $L$ is a proper subspace of $\R^n$ with orthogonal complement
$L^\perp$. In order for this integral to exist we have to assume $q>\dim
L$. Observe that
\begin{align}\label{eq: connection slicing function dcm}
    \tilde{C}_{K,\varphi,q}(\sph^{n-1})=\int_{K \vert L}g_{K,L,\varphi,q}(x) \, d \mathcal{H}^{\dim L}(x).
\end{align}In the next two lemmas we collect some basic properties of the
function which enable us to apply a divergence theorem later on. 
\begin{lemma}\label{lem: g contninuous interior}
Let $K \in \Knull$, $L \subset \mathbb{R}^n$ be a proper subspace and $q
> \dim L$.  Let $\varphi \in \HomF{q-n}$ be Lipschitz continuous on
$\sph^{n-1}$. 
Then 
\begin{enumerate}
    \item $g_{K,L,\varphi,q}$ is bounded on $K \vert L$.
    \item $g_{K,L,\varphi,q}$ is upper semicontinuous in $K \vert L
     $. 
    \item For $x \in K \vert L$ it holds
      \begin{align*}
        \lim_{m \to \infty}g_{K,L,\varphi,q}(e^{-\frac{1}{m}}x)=g_{K,L,\varphi,q}(x).
    \end{align*}
\end{enumerate}
\end{lemma}

\begin{proof}
For $x\in K|L$ let  $K_x= K \cap (x+L^{\perp})$  and let $k=\dim
L$.

For i) let $R>0$ such that $K\subseteq R\,B_n$, and let
$\alpha\in\R_{>0}$ such that 
$\varphi(v)\leq \alpha$ for all $v\in\sph^{n-1}$. 
Applying spherical
coordinates with respect to an orthonormal basis in $L\cup L^\perp$ we
can write
\begin{equation*}
  \begin{split}
    g_{K,L,\varphi,q}(x)&=\int_{K_x}\varphi(z) \, d\mathcal{H}^{n-k}(z)
  \\ &\leq \int_{R\,B_n\cap L^\perp}\varphi(z+x) \,
  d\mathcal{H}^{n-k}(z)\\
  &= \int_{R\,B_n\cap L^\perp}\vert
  z+x\vert^{q-n}\varphi\left(\frac{z+x}{\vert z+x\vert}\right) \,
    d\mathcal{H}^{n-k}(z)\\
  &\leq  \alpha \int_{R\,B_n\cap L^\perp}\vert
  z+x\vert^{q-n} \,
    d\mathcal{H}^{n-k}(z).
  \end{split} 
\end{equation*}
For $q\geq n$ the integrand is bounded 
and so it is $ g_{K,L,\varphi,q}(x)$. So let $q<n$. As $x$ and $z$ are
contained in orthogonal subspaces we have $\vert z+x\vert \geq \vert z \vert$ and
so we may write
\begin{equation*}
  \begin{split}
    g_{K,L,\varphi,q}(x)&\leq \alpha \int_{R\,B_n\cap L^\perp}\vert
  z\vert^{q-n} \,
    d\mathcal{H}^{n-k}(z).
  \end{split}
\end{equation*}   
Since $q>k$, the integral is bounded.

In order to show ii), let $x\in K|L$ and $y_m\in K|L$, $m\in\N$,  with
$\lim_{m\to\infty} y_m=x$.   By the Blaschke selection theorem we can
assume that the sequence $C_m=K_{y_m}-y_m \subset L^{\perp}$ 
converges to a compact convex set $C \subset L^{\perp}$ with respect
to the Hausdorff distance. Thus $K_{y_m}$ converges to $x+C$ with
$x+C\subseteq K\cap (x+L^\perp)= K_x$.
Then we obtain by the Lebesgue's dominated convergence theorem  \begin{align*}
   \lim_{m \to \infty}g_{K,L,\varphi,q}(y_m)&=\lim_{m \to \infty}\int_{K_{y_m}} \varphi(z) \,d \mathcal{H}^{n-k}(z)\\
    &=\int_{x+C} \varphi(z)  \,d \mathcal{H}^{n-k}(z)\\
    &\leq \int_{K_x} \varphi(z)  \,d \mathcal{H}^{n-k}(z)\\&=g_{K,L,\varphi,q}(x).
\end{align*}
Finally, we come to iii).  As $0 \in \operatorname{int} K$ it holds
$e^{-\frac{1}{m}}K_x \subseteq K_{e^{-\frac{1}{m}}x}$ for
$m\in\N_{\geq 1}$. Thus
\begin{align*}
    g_{K,L,\varphi,q}(e^{-\frac{1}{m}}x)&=\int_{K_{e^{-\frac{1}{m}}x}} \varphi(z) \,d \mathcal{H}^{n-k}(z)\\
    &\geq \int_{e^{-\frac{1}{m}}K_{x}} \varphi(z) \,d \mathcal{H}^{n-k}(z)
    \\ &=e^{-\frac{q-k}{m}}\int_{K_x} \varphi(z) \,d \mathcal{H}^{n-k}(z)\\&=e^{-\frac{q-k}{m}}g_{K,L,\varphi,q}(x).
\end{align*}
Hence $g_{K,L,\varphi,q}(x)\leq \lim_{m \to
  \infty}g_{K,L,\varphi,q}(e^{-\frac{1}{m}}x)$ and 
combined with ii) the claim follows.
\end{proof}

Next we want to study Lipschitz continuity and differentiability properties of
$g_{K,L,\varphi,q}(x)$. To this end we need the following lemma.
\begin{lemma} \label{lem : technical} Let $K \in \Knull$,  $L\subset\R^n$ be a proper
  subspace, $q>\dim L$ and  let $\varphi \in \HomF{q-n}$ be Lipschitz
  continuous on $\sph^{n-1}$. For $x\in K|L$ let
  $K_x=K\cap(x+L^\perp)$ and let $U(x, \varepsilon)=x+(\varepsilon B_n\cap
  L)$ for an $\varepsilon>0$.
  \hfill
  \begin{enumerate}
  \item For $x\in\inte K|L$ there exists $\ov \varepsilon_x>0$ and a constant
    $\ov c_x$ depending on $x$ such that for all $x_1,x_2\in
    U(x, \ov \varepsilon_x)$ with $|x_1|\geq |x_2|$ 
       \begin{equation}\label{eqn:summand1}
\Big \vert \int_{K_{x_1}} \varphi(z) \, d \mathcal{H}^{n-k}(z) -
\int_{K_{x_2}+(x_1-x_2)} \varphi(z) \, d \mathcal{H}^{n-k}(z)\Big\vert \leq
{\ov c_x}|x_1-x_2|.
\end{equation} 
   \item  For $x\in\inte K|L\setminus\{0\}$ there exists $ \varepsilon_x>0$ and a constant
    $c_\varphi$ depending only on $\varphi$  such that for all $x_1,x_2\in
    U(x,\varepsilon_x)$
\begin{equation}\label{eqn:summand2}  
  \begin{split} 
     \Bigg \vert\int_{K_{x_2}} & \varphi(z) \, d \mathcal{H}^{n-k}(z) - \int_{K_{x_2}+x_1-x_2} \varphi(z) \, d \mathcal{H}^{n-k}(z)\Bigg \vert \\      
     &\leq \left(c_\varphi\,
      \int_{-x_2+K_{x_2}} |z+x|^{q-n-1}  \, d \mathcal{H}^{n-k}(z)\right) \vert x_1-x_2 \vert.
\end{split}
\end{equation}
\end{enumerate}
\label{lem:twosums}
\end{lemma} 
\begin{proof}  Let $k=\dim L$. Further, for abbreviation we write   $K_i=-x_i+K_{x_i}\subseteq L^\perp$ for
$i=1,2$. Let $R>0$ such that $(1/R)B_n
  \subseteq K \subseteq RB_n$, and let $\alpha\in\R_{\geq 0}$ such that $\varphi(v)\leq
\alpha$ for all $v \in\sph^{n-1}$.

For i) observe that  
\begin{align*}
  \Big \vert \int_{K_{x_1}} \varphi(z) \, d \mathcal{H}^{n-k}(z) &-
\int_{K_{x_2}+(x_1-x_2)} \varphi(z) \, d
  \mathcal{H}^{n-k}(z)\Big\vert\\  
     =\Big|\int_{K_1+x_1} &\varphi(z) \, d \mathcal{H}^{n-k}(z)-
                                 \int_{K_2+x_1} \varphi(z) \, d
                                 \mathcal{H}^{n-k}(z)\Big|\\
  & \leq \int_{K_1 \setminus K_2\cup K_2\setminus K_1} \varphi(z+x_1)
         \,  d \mathcal{H}^{n-k}(z)\\
  & \leq\alpha  \int_{K_1 \setminus K_2\cup K_2\setminus K_1}  \vert z+x_1 \vert^{q-n} \,  d \mathcal{H}^{n-k}(z).
\end{align*}
Now we claim that $\vert z+x_1 \vert^{q-n}\leq \max\{(2R)^{q-n},
R^{n-q}\}$ for $z\in K_1 \setminus K_2\cup K_2\setminus K_1$:

If $q\geq n$ this follows from $K\subseteq R\,B_n$.
So let $q<n$, and
suppose $\vert z+x_1 \vert^{q-n}\geq R^{n-q}$, i.e.,  $\vert z+x_1\vert
\leq 1/R$. Since $z$ and $x_1$ are contained in orthogonal subspaces and since
$|x_1|\geq |x_2|$ we also have $\vert z+x_2 \vert\leq 1/R$.  
As $(1/R)B_n\subseteq K$ this implies $z+x_i \in K\cap
(x_i+L^\perp)=K_{x_i}$, $i=1,2$, and we get the contradiction $z\in
K_1\cap K_2$.

We conclude 
\begin{align*}
   \Big \vert \int_{K_1+x_1} \varphi(z) \, d \mathcal{H}^{n-k}(z) &-
     \int_{K_2+x_1} \varphi(z) \, d \mathcal{H}^{n-k}(z)\Big\vert
  \\ \leq & \alpha\,\max \{(2R)^{q-n}, R^{n-q}\} \vol_{n-k}(K_1
            \setminus K_2\cup K_2 \setminus K_1).
\end{align*}
By a result of Groemer \cite[Theorem. i)]{Groemer2000} on comparing
different metrics on the space of convex bodies we have 
\begin{equation*}
       \vol_{n-k}(K_1
            \setminus K_2\cup K_2 \setminus K_1)\leq c(n,K)\,d(K_1,K_2).
          \end{equation*}
where  $c(n,K)$ is s a constant depending only on $n$ and $K$. 
On the other hand, according to \cite[Lemma 2.3]{klee1960polyhedral} the Hausdorff
distance of sections of convex
bodies is locally Lipschitz continuous, i.e., there exists
$\ov \varepsilon_x>0$ and a constant $c_x>0$ such that $d(K_{x_1},K_{x_2})\leq
c_x|x_1-x_2|$ for all $x_1,x_2\in U(x, \ov \varepsilon_x)$. As $d(K_1,K_2)\leq
d(K_{x_1},K_{x_2})$ we have shown

\begin{equation}
\Big \vert \int_{K_{x_1}} \varphi(z) \, d \mathcal{H}^{n-k}(z) -
\int_{K_{x_2}+(x_1-x_2)} \varphi(z) \, d \mathcal{H}^{n-k}(z)\Big\vert \leq
{\ov c_x}|x_1-x_2|
\label{eq:summand1}
\end{equation} 
for all $x_1,x_2\in U(x, \ov \varepsilon_x)$ and a suitable constant
${\ov c_x}$.

Now we come to ii) and here we assume $x\in\inte
K|L\setminus\{0\}$. First we note that 
\begin{equation*}
  \begin{split} 
     \Bigg \vert\int_{K_{x_2}} \varphi(z) \, d \mathcal{H}^{n-k}(z) &- \int_{K_{x_2}+(x_1-x_2)} \varphi(z) \, d \mathcal{H}^{n-k}(z)\Bigg \vert \\  = & \left \vert \int_{K_2} \varphi ( z +  x_2) \, d \mathcal{H}^{n-k}(z)- \int_{K_2} \varphi ( z+ x_1 ) \, d \mathcal{H}^{n-k}(z)\right \vert \\
      \leq & \int_{K_2} \vert\varphi ( z +  x_2)-\varphi ( z +
      x_1) \vert \, d \mathcal{H}^{n-k}(z).
  \end{split}     
\end{equation*}   
Let now $\varepsilon_x=\frac{1}{2}|x|$, and $x_1,x_2\in
U(x,\varepsilon_x)$. Then for $z\in L^\perp$ we have  $z+x_1,z+x_2
\in z+x +\frac{1}{2}|x|B_n\subseteq z+x
+\frac{1}{2}|z+x|B_n$. In view of  Lemma  \ref{lem: varphi loc Lipsch})
we get
\begin{equation*}
  \begin{split} 
  \int_{K_2} \vert\varphi ( z +  x_2)&-\varphi ( z +
      x_1) \vert \, d \mathcal{H}^{n-k}(z) \\ &\leq  c_\varphi\,|x_1-x_2|
      \int_{K_2} |z+x|^{q-n-1}  \, d \mathcal{H}^{n-k}(z),
   \end{split}    
\end{equation*}   
where $c_\varphi$ is a constant depending on $\varphi$.
\end{proof}

\begin{prop} \label{prop : g Lip, diff, grad bound} Let $K \in \Knull$,  $L\subset\R^n$ be a proper
  subspace, $q>\dim L$, and  let $\varphi \in \HomF{q-n}$ be Lipschitz
  on $\sph^{n-1}$.
  \begin{enumerate}
   \item $g_{K,L,\varphi,q}(x)$ is locally
  Lipschitz in $\operatorname{int}(K \vert
  L)\setminus\{0\}$.
\item  $g_{K,L,\varphi,q}(x)$ is almost
  everywhere differentiable in $\operatorname{int}(K \vert
  L)$.
\item Let $q>\dim L+1$. Then

  \begin{equation*}
    \int_{K \vert L} \vert \langle \nabla g_{K,L,\varphi,q}(x),x
    \rangle \vert \, d \mathcal{H}^{\dim L}(x) < \infty.
  \end{equation*}
  \end{enumerate}
\end{prop}   
\begin{proof} The second statement follows directly from i) via
  Rademacher's theorem (cf. Theorem 3.1.6 in
  \cite{federer2014geometric}). In order to verify i) we use Lemma
  \ref{lem : technical} and its notation. So let $x\in\inte K|L\setminus\{0\}$,
  $\delta_x=\min\{\varepsilon_x, \ov \varepsilon_x\}$ and let
  $x_1,x_2\in U(x,\delta_x)$ and assume $|x_1|\geq |x_2|$. Then
  \begin{equation}
  \begin{split}
   \lvert g_{K,L,\varphi,q}(x_1) &-g_{K,L,\varphi,q}(x_2) \rvert  \\ &= \left \vert \int_{K_{x_1}} \varphi(z) \, d \mathcal{H}^{n-k}(z)- \int_{K_{x_2}} \varphi(z) \, d \mathcal{H}^{n-k}(z)\right \vert \\
    &\leq  \left \vert \int_{K_{x_1}} \varphi(z) \, d
      \mathcal{H}^{n-k}(z)- \int_{K_{x_2}+(x_1-x_2)} \varphi(z) \, d
      \mathcal{H}^{n-k}(z)\right \vert \\ &\quad\quad\quad\quad +
    \left \vert\int_{K_{x_2}} \varphi(z) \, d \mathcal{H}^{n-k}(z)-
      \int_{K_{x_2}+(x_1-x_2)} \varphi(z) \, d
      \mathcal{H}^{n-k}(z)\right \vert \\
    &\leq |x_1-x_2|\left(\ov c_x +  c_\varphi
      \int_{-x_2+K_{x_2}} |z+x|^{q-n-1}  \, d
      \mathcal{H}^{n-k}(z)\right). 
  \end{split}
  \label{eq:summands}    
\end{equation}
Again assuming that $K\subseteq RB_n$ we may bound
\begin{equation}
  \begin{split}
   \lvert g_{K,L,\varphi,q}(x_1) & -g_{K,L,\varphi,q}(x_2) \rvert  \\
   & \leq |x_1-x_2|\left(\ov c_x +  c_\varphi
      \int_{RB_n\cap L^\perp} |z+x|^{q-n-1}  \, d\mathcal{H}^{n-k}(z)\right). 
 \end{split}
\end{equation}  
As $|z+x|\leq 2R$ the last integral is bounded by
$\widetilde{c}=(2R)^{q-k-1}\vol_{n-k}(B_n\cap L^\perp)$ if $q-n-1\geq 0$.  If $q-n-1<0$
we note that $|z+x|\geq |x|>0$ and so 
\begin{equation*}
  \begin{split} 
        \int_{RB_n\cap L^\perp} |z+x|^{q-n-1}  \, d \mathcal{H}^{n-k}(z)\leq
        \widetilde c_x= |x|^{q-n-1} R^{n-k} \vol_{n-k}(B_n\cap
        L^\perp).
\label{eq:secondsum}        
    \end{split}     
      \end{equation*}
Altogether we obtain, 
\begin{equation}
  \begin{split}
   \lvert g_{K,L,\varphi,q}(x_1)  -g_{K,L,\varphi,q}(x_2) \rvert  
   \leq |x_1-x_2|\left(\ov c_x +
     c_\varphi\max\{\widetilde{c},\widetilde{c}_x\}\right),
 \label{eq:locallip}  
 \end{split}
\end{equation} 
 which shows i).

In order to verify  iii) we will first argue that in the case $x\in\inte
e^{-\frac{1}{m}}K|L\setminus\{0\}$ and $q>\dim L +1$ we can make the 
constants in \eqref{eq:locallip} independent of $x$.

We start with $\widetilde{c}_x$ from   \eqref{eq:locallip} appearing
in the case  $q-n-1<0$. If $q>\dim
L+1$ and so $q-n-1>\dim L-n$ the integral $\int_{RB_n\cap L^\perp}
|z+x|^{q-n-1}  \, d \mathcal{H}^{n-k}(z)$ is bounded from above by a 
constant $\widetilde c'$ for any $x\in K|L$. Hence, \eqref{eq:locallip} becomes
\begin{equation}
  \begin{split}
   \lvert g_{K,L,\varphi,q}(x_1)  -g_{K,L,\varphi,q}(x_2) \rvert  
   \leq |x_1-x_2|\left(\ov c_x +
     c_\varphi\max\{\widetilde{c},\widetilde{c}'\}\right),
 \label{eq:locallipa}  
 \end{split}
\end{equation} 
for all $x_1,x_2\in U(x,\delta_x)$. 
By a standard compactness argument we can bound the constants $c_x$ for all $x \in e^{-\frac{1}{m}}K|L$ by a constant  $\overline{c}$ and so we get
\begin{equation}
  \begin{split}
   \lvert g_{K,L,\varphi,q}(x_1)  -g_{K,L,\varphi,q}(x_2) \rvert  
   \leq |x_1-x_2|\left(\ov c +
     c_\varphi\max\{\widetilde{c},\widetilde{c}'\}\right),
 \label{eq:locallipb}  
 \end{split}
\end{equation}
for all $x_1,x_2\in U(x,\delta_x)$, and $x\in
e^{-\frac{1}{m}}K|L\setminus\{0\}$.
Hence, for any $x \in e^{-\frac{1}{m}}K \vert L \setminus \{0\}$ where
$\nabla g_{K,L,\varphi,q} (x)$ exists, it holds
\begin{equation*}
  \begin{split}
    \vert \langle \nabla g_{K,L,\varphi,q}(x), x  \rangle \vert &=
    \lim_{\varepsilon \to 0} \frac{\vert
      g_{K,L,\varphi,q}(x+\varepsilon x) -g_{K,L,\varphi,q}(x) \vert
    }{|\varepsilon|} \\ & \leq \left(\ov c +
     c_\varphi\max\{\widetilde{c},\widetilde{c}'\}\right)|x|.
  \end{split}   
  \end{equation*} 
According to ii) the gradient exists almost everywhere in $K|L$ and so
we have
\begin{equation*}
  \begin{split} 
  \int_{K \vert L}&|\langle \nabla g_{K,L,\varphi,q}(x),x \rangle| \, d
  \mathcal{H}^{\dim L}(x) \\
& = \lim_{m \to \infty } \int_{e^{-\frac{1}{m}}K \vert L}|\langle
        \nabla g_{K,L,\varphi,q}(x),x \rangle | \, d \mathcal{H}^{\dim
          L}(x) < \infty.
     \end{split}    \qedhere
    \end{equation*}
\end{proof}

Regarding Propostion \ref{prop : g Lip, diff, grad bound} i) we remark that in general $g_{K,L,\varphi,q}$ is not locally Lipschitz in 0, as the following example shows:
 Let $k \in \{1, \dotsc, n-1\}$ and $\varphi( \cdot)= \vert \cdot \vert^{q-n}$, denote $K=B_k \times B_{n-k} \subset \R^n$ and
$L= \R^k$. Note that the sections of $K$ are the same up to translation. Therefore the summand in \eqref{eqn:summand1} is zero and only the term in \eqref{eqn:summand2} is relevant to decide local Lipschitz continuity in zero.
  Let $k+1>q >k$ and $x  \in K \vert L$ with $\vert x \vert <1 $. As before we obtain
  \begin{align*}
      g_{K,L,q}(x)=& 
       \int_{B_{n-k}}\vert x+z  \vert^{q-n}  \, d \mathcal{H}^{n-k}(z) \\
      =& \int_0^1 \int_{\sph^{n-k-1}}r^{n-k-1}\vert x+ru  \vert^{q-n} \, du \,dr \\
      =&(n-k) \text{vol}(B_{n-k}) \vert x \vert^{q-k} \int_0^{\frac{1}{\vert x \vert}} s^{n-k-1} \sqrt{ s ^2 +1}^{q-n} \, ds \\
      \leq & (n-k) \text{vol}(B_{n-k})  \vert x \vert^{q-k} \left( \left(\frac{1}{\sqrt{2}}\right)^{n-q} \int_0^1 s^{q-k-1} \, ds + \int_1^{\frac{1}{\vert x \vert}}s^{q-k-1} \, ds\right) \\
      \leq &(n-k) \text{vol}(B_{n-k})   \frac{1}{q-k}\left(\left(\frac{1}{\sqrt{2}}\right)^{n-q}\vert x \vert^{q-k} +1-\vert x \vert^{q-k} \right).
  \end{align*}
On the other hand it holds
\begin{align*}
       g_{K,L,q}(0)
      =& \int_{B_{n-k}} \vert z \vert^{q-n} \, d \mathcal{H}^{n-k}(z) \\
      =& \int_0^1 \int_{\sph^{n-k-1}}r^{q-k-1} \, du \,dr \\  
      =&(n-k) \text{vol}(B_{n-k})  \frac{1}{q-k}.
\end{align*}
This gives
\begin{align*}
    &\frac{ \vert g_{K,L,q}(x)- g_{K,L,q}(0) \vert}{\vert x \vert } \\ \geq & (n-k) \text{vol}(B_{n-k})  \frac{1}{q-k} \vert x \vert ^{q-k-1} \left(1- \left(\frac{1}{\sqrt{2}}\right)^{n-q}\right), 
\end{align*}
which goes to infinity as $\vert x \vert$ goes to zero.

The next lemma which gives a representation of $\dcura_{K,\varphi,q}( \eta)$
via the inverse Gauss map was proven for $\varphi=|\cdot|^{q-n}$ in
\cite{huang2016geometric} and our slightly more general case can be
done analogously.  

\begin{lemma} \label{lem: Cor 3.5 HLYZ}
Let $K \in \Knull$, $q>0$. Let $\varphi \in \HomF{q-n}$ be Lipschitz continuous on $\sph^{n-1}$. Let further $\eta \subseteq \sph^{n-1}$ be a Borel set. Then \begin{equation*} \label{eq: dualcurvat via nuFunkt}
\dcura_{K,\varphi,q}( \eta) = \frac{1}{n}
\int_{\nu_K^{-1}(\eta)} \langle\nu_K(x), x\rangle \varphi
(x)\, d \mathcal{H}^{n-1}(x).
\end{equation*}
\end{lemma}
\begin{proof} See \cite[Lemma 3.5]{huang2016geometric}.
\end{proof}   
We aim to express the dual curvature measure on a subspace as a multiple of the dual curvature measure on the whole sphere plus a term depending on the directional derivative of $g_{K,L,\varphi, q}$. This can be established via a divergence theorem following the approach presented in \cite{boroczky2014conevolume,henk2014cone}. The following divergence theorem presented in \cite{pfeffer2012divergence} will be employed.
\begin{theo}\label{theo: GG div}
Let $A \in \Kn$. Let $G: A \to \mathbb{R}^n$ be a bounded vector field, which is locally Lipschitz continuous in $A \setminus \{a\}$ for some $a \in A$. Furthermore suppose that $\operatorname{div}G \in L_1(A)$. Then it holds
\begin{align*}
    \int_{A} \operatorname{div} G(x) \, d \mathcal{H}^{n}(x) = \int_{\partial^{\ast} A} \langle G(x), \nu_A(x) \rangle \, d \mathcal{H}^{n}(x).
\end{align*}
\end{theo}
This is a special case of the divergence theorem presented in Proposition 7.4.3 in \cite{pfeffer2012divergence} for sets of bounded variation and admissable vector fields. Locally Lipschitz vector fields and convex bodies satisfy these presumptions. \\
The next theorem is to some extend our main result as it relates
$\dcura_{K,\varphi,q}(\sph^{n-1}\cap L)$ to  $\dcura_{K,\varphi,q}(
\sph^{n-1})$ via Theorem \ref{theo: GG div}. All further results are
based on this.  

\begin{theo}\label{thm: Gauss div gen}
Let $K \in \Knull$, $L\subset \Rn$ be a proper subspace and $q > \dim
L$. Let $\varphi \in \HomF{q-n}$ be Lipschitz continuous on
$\sph^{n-1}$. Furthermore assume that for any $m \in \mathbb{N}_{\geq 1}$
\begin{equation*} 
\int_{e^{-\frac{1}{m}}K \vert L} \vert \langle \nabla
g_{K,L,\varphi,q}(x),x \rangle \vert \, d \mathcal{H}^{\dim L}(x) <
\infty.
\end{equation*}
Then it holds
\begin{align*}
    &\rat(\dcura_{K,\varphi,q},L)  \\= &\frac{\dim L}{q} +\frac{1}{n}\frac{1}{ \dcura_{K,\varphi, q}(\sph^{n-1})}
 \lim_{m \to \infty}\int_{e^{-1/m}K \vert L}\ip{\nabla
   g_{K,L,\varphi,q}(x)}{x} \, \dint \haus^{\dim L}(x).
\end{align*}
\end{theo}

\begin{proof}
We follow the proof outline of Lemma 3.3 in
\cite{boroczky2014conevolume}. Let $\dim L=k$, and for $m \in
\mathbb{N}_{\geq 1}$ we set
\[ E_m= e^{-\frac{1}{m}} K \vert L. 
  \]
The relative boundary of $E_m$ with respect to the subspace $L$ will be denoted by $\bar{\partial}E_m$. We define the vector field $G: K \vert L \to \mathbb{R}^n$ by \[G(x)=g_{K,L,\varphi,q}(x)x.\]
By Proposition \ref{prop : g Lip, diff, grad bound}  ii) $G$ is almost everywhere differentiable on
$K|L$ and for its divergence $\operatorname{div}G$ we find 
\begin{equation} \label{eqn: div Formula}
\operatorname{div}G(x) = k g_{K,L,\varphi,q}(x)+ \langle \nabla
g_{K,L,\varphi,q}(x),x \rangle.
\end{equation}
Hence with (\ref{eq: connection slicing function dcm}) it follows
\begin{equation*}
  \begin{split} 
 \int_{E_m} &\vert \operatorname{div} G(x) \vert \,  d
 \mathcal{H}^{k}(x) \\ &\leq k\frac{n}{q}
 \tilde{C}_{K,\varphi,q}(\sph^{n-1})+ \int_{E_m} \vert \langle
 \nabla g_{K,L,\varphi,q}(x), x \rangle \vert \, d \mathcal{H}^{k}(x)
 < \infty.
 \end{split} 
\end{equation*}
So it holds $\operatorname{div} G\in L_1(E_m)$. Thus we may
apply the divergence theorem, i.e., Theorem \ref{theo: GG div} and get 
\begin{equation}\label{eqn: DivTheo for G on Em} 
    \int_{E_m} \operatorname{div}G(x) \, d \mathcal{H}^{k}(x) = \int_{\bar{\partial}^{*}E_m} \langle G(x), \nu_{E_m}(x) \rangle \, d \mathcal{H}^{k-1}(x).
\end{equation}
First we consider the right-hand side of \eqref{eqn: DivTheo for G on
  Em}.  For a regular boundary point $x \in
\bar{\partial}^{*}K \vert L$ we have  $\nu_{K \vert
  L}(x)=\nu_{E_m}(e^{-\frac{1}{m}}x)$ and so we  may write 
\begin{align*}
    \int_{\bar{\partial}^{*}E_m} \langle G(x), \nu_{E_m}(x) \rangle \, d \mathcal{H}^{k-1}(x)&=e^{-\frac{k-1}{m}}\int_{\bar{\partial}^*K \vert L} \langle G(e^{-\frac{1}{m}}x), \nu_{K \vert L}(x) \rangle \, d \mathcal{H}^{k-1}(x) \\
    &=e^{-\frac{k}{m}}\int_{\bar{\partial}^*K \vert L} g_{K,L,\varphi,q}(e^{-\frac{1}{m}}x)\langle x, \nu_{K \vert L}(x) \rangle \, d \mathcal{H}^{k-1}(x).
\end{align*}
By Lemma \ref{lem: g contninuous interior} iii)
we have $g_{K,L,\varphi,q}(e^{-\frac{1}{m}}x)\to g_{K,L,\varphi,q}(x)$
pointwise for $m \to \infty$ and with the Lebesgue dominated convergence theorem and
the 
definition of $g_{K,L,\varphi,q}(x)$ we
obtain
\begin{align*}
    \lim_{m \to \infty} \int_{\bar{\partial}^{*}E_m} & \langle G(x), \nu_{E_m}(x) \rangle \, d \mathcal{H}^{k-1}(x) \\
    =& \int_{\bar{\partial}^{*}K \vert L} g_{K,L,\varphi,q}(x)\langle x,
       \nu_{K \vert L}(x) \rangle \, d \mathcal{H}^{k-1}(x)\\
  =& \int_{\bar{\partial}^{*}K \vert L} \int_{K\cap (x+L^\perp)} \varphi(z) \langle x, \nu_{K \vert L}(x) \rangle \, d \mathcal{H}^{n-k}(z)\, d \mathcal{H}^{k-1}(x).
\end{align*}
Now set  $M=\partial K \cap(L^{\perp}+\bar{\partial}^{*}K\vert
L)$. Then the set of regular points in $M$ is precisely the set of all regular boundary points of
$K$ having their unique outer normal vector in $L\cap
\sph^{n-1}$. In view of Lemma \ref{lem: Cor 3.5 HLYZ} we get

\begin{align} \label{eqn: expr 1}
     \lim_{m \to \infty} \int_{\bar{\partial}^{*}E_m} \langle G(x), &\nu_{E_m}(x) \rangle \, d \mathcal{H}^{k-1}(x)\\
     =&\int_{M} \varphi(z)\langle z \vert L, \nu_{K \vert L}(z \vert L) \rangle \, d \mathcal{H}^{n-1}(z) \\
     =&\int_{\nu_K^{-1}(L\cap \sph^{n-1})} \varphi(z)\langle z, \nu_{K }(z) \rangle \, d \mathcal{H}^{n-1}(z) \\
     = & n \tilde{C}_{K,\varphi,q}(L \cap \sph^{n-1}).
\end{align}
Next we turn to the left-hand side of \eqref{eqn: DivTheo for G on Em}
which by \eqref{eqn: div Formula} is 
\begin{equation} \label{eqn: div form int}
 \begin{split}  
    \int_{E_m} &\operatorname{div}G(x) \, d \mathcal{H}^{k}(x) \\&=
    k\int_{E_m}g_{K,L,\varphi,q}(x) \, d \mathcal{H}^{k}(x) +
    \int_{E_m} \langle \nabla g_{K,L,\varphi,q}(x), x\rangle \, d
    \mathcal{H}^{k}(x).
\end{split}     
\end{equation} Again by the Lebesgue
dominated convergence theorem it
holds \begin{equation}
  \begin{split} 
    \lim_{m \to \infty} \int_{E_m}g_{K,L,\varphi,q}(x) \, d
    \mathcal{H}^{k}(x) & = \int_{K \vert L}g_{K,L,\varphi,q}(x) \, d
    \mathcal{H}^{k}(x) \\ &= \frac{n}{q} \tilde{C}_{K,\varphi,q}(
    \sph^{n-1}).\label{eqn: sequence 2 conv}
    \end{split} 
  \end{equation}
  Hence, by \eqref{eqn: DivTheo for G on Em} and  \eqref{eqn: expr 1}
 we finally get 
\begin{align*}
   \lim_{m \to \infty} &\int_{E_m} \langle \nabla g_{K,L,\varphi,q}(x), x\rangle \, d \mathcal{H}^{k}(x) \\ =&  \lim_{m \to \infty} \int_{\bar{\partial}^{*}E_m} \langle G(x), \nu_{E_m}(x) \rangle \, d \mathcal{H}^{k-1}(x) -k  \lim_{m \to \infty} \int_{E_m}g_{K,L,\varphi,q}(x) \, d \mathcal{H}^{k}(x) \\ =& n \tilde{C}_{K,\varphi,q}( L \cap \sph^{n-1})-k\frac{n}{q} \tilde{C}_{K,\varphi,q}(\sph^{n-1}). \qedhere
\end{align*}
\end{proof} 

Combined with Lemma \ref{prop : g Lip, diff, grad bound} iii) we
obtain a slight generalization of Theorem \ref{thm:Gaussdiv} as a corollary:
\begin{corollary}\label{cor : Gauss div thm q>k+1}
Let $K \in \Knull$, $L\subset \Rn$ be a proper subspace and $q > \dim L+1$. Let $\varphi \in \HomF{q-n}$ be Lipschitz continuous on $\sph^{n-1}$. Then it holds
\begin{align*}
    n\dcura_{K,\varphi,q}(\sph^{n-1}\cap L)
  &=\frac{n}{q}\dim(L)\dcura_{K,\varphi,q}( \sph^{n-1})\\ &+\int_{K \vert L}\langle \nabla g_{K,L,{\varphi},q}(x), x \rangle \, d \mathcal{H}^{\dim L}(x).
\end{align*}
\end{corollary}


\section{Proof of Theorem \ref{theo:maingeneralbodies} and
  \ref{thm:boundssection}} \label{sec : 4}
Here we return to the function $\varphi=|\cdot|^{q-n}$ and
depending on $q$ we have to distinguish the quasiconcave range
$q\leq n$ and the convex range  $q\geq n+1$ of this
function. The quasiconvex range $n<q<n+1$ remains open.  

Now let $K\in\Knull$, $\gamma\in [0,1]$ such that $\gamma (-K)\subseteq
K$. In order to exploit this fact for our purposes
we note that for any $x\in K|L$
\begin{equation*}
   -\gamma \left(K\cap (x+L^\perp)\right) \subseteq K\cap \left(-\gamma x+L^\perp\right).
\end{equation*}   
 Hence for any $\lambda\in [0,1]$ we have
\begin{equation} \label{eqn:inclusion_centered} 
  \left(\frac{\gamma+\lambda}{1+\gamma}\right) K_x
  +\left(\frac{(1-\lambda)\gamma}{1+\gamma}\right) (-K_{x})\subseteq
  K_{\lambda x} 
\end{equation} 
where we set $K_y=K\cap (y+L^\perp)$ for $y\in K|L$. Observe that the
left-hand side is in general strictly larger than $\lambda K_x$ which
is contained in $K_{\lambda x}$ as $0\in K$.



\subsection{The quasiconcave range $q\leq n$}
First we state a lemma from \cite{boroczky2016cone} in a different but
equivalent form.  

\begin{lemma}\label{lem: Anderson for quasiconcave}
Let $K \in \Kn$ with $\dim(K)=k$. Let $\varphi \in \HomF{p}$ be quasiconcave, even and integrable on $k$-dimensional compact convex sets. Then it holds for $\lambda_0,\lambda_1 > 0$ 
$$ \int_{\lambda_0 K+ \lambda_1 (-K)} \varphi(z) \, d \mathcal{H}^{k}(z) \geq (\lambda_0+\lambda_1)^{p+k} \int_K \varphi(z) \, d \mathcal{H}^{k}(z) .$$ \end{lemma}
\begin{proof}
Setting $\lambda=\frac{\lambda_0}{\lambda_0+\lambda_1}\in[0,1]$ we get 
\begin{align*}
    \int_{\lambda_0 K+ \lambda_1 (-K)} \varphi(z) \, d \mathcal{H}^{k}(z) =& (\lambda_0+\lambda_1)^{p+k}  \int_{\lambda K+ (1-\lambda) (-K)} \varphi(z) \, d \mathcal{H}^{k}(z) \\
    \geq& (\lambda_0+\lambda_1)^{p+k} \int_K \varphi(z) \, d \mathcal{H}^{k}(z),
\end{align*}
where for the last inequality we use Lemma
3.1. in \cite{boroczky2016cone}. 
\end{proof}

\begin{lemma}\label{lem: bound on gradient, quasiconcave case}
   Let $K \in \Knull$, $\gamma\in [0,1]$ such that $\gamma (- K) \subseteq
  K$, $L \subset\mathbb{R}^n$ be a proper subspace and let
  $\dim(L)< q\leq n$. Let $x\in\inte K|L$ and assume that  $\nabla
  g_{K,L,q}(x)$ exists. Then it holds
\begin{align*}
    \langle \nabla g_{K,L,q}(x),x \rangle \leq
  \frac{1-\gamma}{\gamma+1}\, (q-\dim(L))\,g_{K,L,q}(x).
\end{align*}
\end{lemma}
\begin{proof}
For $y\in K|L$ let  $K_y=K \cap (y+L^\perp)$, let $\lambda\in[0,1)$ and
set 
   $\lambda_0=\frac{\gamma +\lambda}{\gamma+1}$ and $\lambda_1=\gamma
   \frac{1-\lambda}{\gamma+1}$. With \eqref{eqn:inclusion_centered} we
   have
   $$
   \lambda_0 K_x + \lambda_1 (-K_x) \subseteq K_{\lambda x}
   $$
and Lemma~\ref{lem: Anderson for quasiconcave} gives 
\begin{align*}
    g_{K,L,q}(\lambda x)=&\int_{K_{\lambda x}} \vert z \vert ^{q-n} \, d \mathcal{H}^{n-\dim L}(z) \\
    & \geq\int_{\lambda_0 K_x + \lambda_1\left(-K_{x}\right)} \vert z \vert ^{q-n} \, d \mathcal{H}^{n-\dim L}(z) \\
    & \geq \left(\lambda_0+\lambda_1\right)^{n-\dim(L)+q-n}\int_{K_x} \vert z \vert ^{q-n} \, d \mathcal{H}^{n-\dim L}(z) \\
    & =  \left(1+(1-\lambda) \frac{\gamma-1}{\gamma + 1}\right)^{q-\dim(L)} g_{K,L,q}(x).
\end{align*}
Setting $\lambda=1-\varepsilon$ yields
\begin{align}\label{eqn: proof grad quasiconcave eq1}
    \frac{g_{K,L,q}((1-\varepsilon)x)-g_{K,L,q}(x)}{\varepsilon} \geq \frac{f(\varepsilon)-f(0)}{\varepsilon} g_{K,L,q}(x),
\end{align}
with $f(t)=\left(\frac{\gamma-1}{\gamma+1}t
  +1\right)^{q-\dim(L)}$. Hence,
\begin{align*}
    \langle \nabla g_{K,L,q}(x), -x \rangle &= \lim_{\varepsilon \to 0}\frac{g_{K,L,q}((1-\varepsilon)x)-g_{K,L,q}(x)}{\varepsilon} \\ 
    &\geq f'(0)\,g_{K,L,q}(x) \\
    &=\frac{\gamma-1}{\gamma+1}\,(q-\dim(L))\,g_{K,L,q}(x). \qedhere
\end{align*}
\end{proof}
\begin{rem}
Lemma \ref{lem: bound on gradient, quasiconcave case} also holds for the section function $g_{K,L,\varphi,q}$ when $\varphi \in \HomF{q-n}$ is quasiconcave and even.
\end{rem}

\subsection{The convex setting in the range $q \geq n+1$}
\begin{lemma}\label{lem: Anderson convex}
Let $K_0, K_1 \in \Knull$ with $\dim(K_i)=k$, $\vol(K_0)=\vol (K_1)$ and assume that their affine hulls are parallel. For $\lambda_0, \lambda_1 >0 $ and $p \geq 1$, it holds
\begin{align*}
    &\int_{\lambda_0 K_0 +\lambda_1K_1} \vert z \vert ^p \, d \mathcal{H}^{k}(z) +
\int_{\lambda_0 K_1 +\lambda_1 K_0} \vert z \vert ^p \, d \mathcal{H}^{k}(z) \\ \geq & (\lambda_0+\lambda_1)^k \vert \lambda_0 -\lambda_1 \vert ^p  \left(\int_{K_0} \vert z \vert ^p \, d \mathcal{H}^{k}(z) + \int_{K_1} \vert z \vert ^p \, d \mathcal{H}^{k}(z)\right).
\end{align*}
\end{lemma}
\begin{proof}
Setting $\lambda= \frac{\lambda_0}{\lambda_0+\lambda_1}$ and applying Theorem 1.9 in \cite{henk2017necessary}, it holds \begin{align*}
    &\int_{\lambda_0 K_0 +\lambda_1K_1} \vert z \vert ^p \, d \mathcal{H}^{k}(z) +
\int_{\lambda_0 K_1 +\lambda_1 K_0} \vert z \vert ^p \, d \mathcal{H}^{k}(z) \\=& (\lambda_0+\lambda_1)^{p+k}\left(\int_{\lambda K_0 +(1-\lambda)K_1} \vert z \vert ^p \, d \mathcal{H}^{k}(z) +
\int_{\lambda K_1 +(1-\lambda) K_0} \vert z \vert ^p \, d \mathcal{H}^{k}(z) \right) \\
\geq &(\lambda_0+\lambda_1)^{p+k} \vert 2 \lambda -1\vert^p  \left(\int_{K_0} \vert z \vert ^p \, d \mathcal{H}^{k}(z) + \int_{K_1} \vert z \vert ^p \, d \mathcal{H}^{k}(z)\right) \\
=&(\lambda_0+\lambda_1)^k \vert \lambda_0 -\lambda_1 \vert ^p  \left(\int_{K_0} \vert z \vert ^p \, d \mathcal{H}^{k}(z) + \int_{K_1} \vert z \vert ^p \, d \mathcal{H}^{k}(z)\right). \qedhere
\end{align*}
\end{proof} 

\begin{lemma}\label{lem: bound on gradient in convex case}
  Let $K \in \Knull$, $\gamma\in [0,1]$ such that $\gamma (-K)\subseteq
  K$, $L \subset\mathbb{R}^n$ be a proper subspace and let
  $q \geq n+1$. Let $x\in\inte K|L$ and assume that  $\nabla
  g_{K,L,q}(x)$ exists. Then it holds
\begin{align*}
    \langle \nabla g_{K,L,q}(x),x \rangle \leq \left((q-n)+  \frac{1-\gamma}{\gamma+1}(n-\dim(L))\right)g_{K,L,q}(x).
\end{align*}
\end{lemma}
\begin{proof}
For $y\in K|L$ let  $K_y=K \cap (y+L^\perp)$, let $\lambda\in[0,1)$ and
set 
   $\lambda_0=\frac{\gamma +\lambda}{\gamma+1}$ and $\lambda_1=\gamma
   \frac{1-\lambda}{\gamma+1}$. With \eqref{eqn:inclusion_centered} we
   have
   $$
   \lambda_0 K_x + \lambda_1 (-K_x) \subseteq K_{\lambda x}
   $$
and Lemma~\ref{lem: Anderson convex} gives for $K_0=K_x$, $K_1=-K_x$
\begin{align*}
    g_{K,L,q}(\lambda x)=&\int_{K_{\lambda x}} \vert z \vert^{q-n} \, d \mathcal{H}^{n-\dim L}(z) \\
    & \geq\int_{\lambda_0 K_x + \lambda_1\left(-K_{x}\right)} \vert z \vert^{q-n} \, d \mathcal{H}^{n-\dim L}(z) \\
    & \geq (\lambda_0+\lambda_1)^{n-\dim(L)}\vert \lambda_0-\lambda_1\vert^{q-n}\int_{K_x} \vert z \vert \, d \mathcal{H}^{n-\dim L}(z) \\
    & = \left(1+(1-\lambda)\frac{\gamma-1}{\gamma+1}\right)^{n-\dim(L)} \lambda^{q-n}g_{K,L,q}(x).
\end{align*}
Setting $\lambda=1- \varepsilon$ it follows \begin{align}\label{eqn: proof grad quasicvx eq2}
    \frac{g_{K,L,q}((1-\varepsilon)x)-g_{K,L,q}(x)}{\varepsilon} \geq \frac{f(\varepsilon)-f(0)}{\varepsilon} g_{K,L,q}(x)
\end{align} with $  f(t) =(1-t)^{q-n}\left(\frac{\gamma-1}{\gamma+1}t +1\right)^{n-\dim(L)}$. Hence,
\begin{align*}
    \langle \nabla g_{K,L,q}(x), -x \rangle &=\lim_{\varepsilon \to 0}\frac{g_{K,L,q}((1-\varepsilon)x)-g_{K,L,q}(x)}{\varepsilon}\\ &\geq  f'(0)g_{K,L,q}(x) \\
    &=\left(-(q-n)+ \frac{\gamma-1}{\gamma+1}(n-\dim(L))\right)g_{K,L,q}(x). \qedhere
\end{align*} 
\end{proof} 
\begin{proof}[Proof of Theorem \ref{thm:boundssection}]
Lemma \ref{lem : generalized dcm representation} combined with Lemma \ref{lem: bound on gradient, quasiconcave case} and Lemma \ref{lem: bound on gradient in convex case} respectively yields the bounds
\begin{align*}
  \int_{K \vert L}\ip{\nabla
   g_{K,L,q}(x)}{x} \, \dint \haus^{\dim L}(x) \leq  c_q \int_{K \vert L} g_{K,L,q}(x)  \, \dint \haus^{\dim L}(x) =c_q\frac{n}{q}\dcura_{K,q}(\sph^{n-1}) \, 
   \end{align*} with
   \begin{align*}
  c_q =\begin{cases}
    &(q-\dim(L))\frac{1-\gamma}{1+\gamma}, q\leq n, 
    \\[1ex]
    &(q-n)+\frac{1-\gamma}{1+\gamma}(n- \dim(L)), q>n+1. \qedhere
   \end{cases}    
\end{align*}
\end{proof}
\begin{proof}[Proof of Theorem \ref{theo:maingeneralbodies}]
The claim follows as a simple corollary of combining Theorem \ref{thm:Gaussdiv} and the bounds for the directional derivative, i.e., Theorem \ref{thm:boundssection}.
\end{proof}
\section*{Acknowledgements}
The authors thank Christian Kipp for his helpful comments and suggestions.
\bibliographystyle{abbrv}
\bibliography{biblio}

\end{document}